\documentclass[11pt]{amsart}

\newtheorem{theorem}{Theorem}[section]

\newtheorem{proposition}[theorem]{Proposition}
\newtheorem{definition}[theorem]{Definition}

\newtheorem{corollary}[theorem]{Corollary}
\newtheorem{remark}[theorem]{Remark}

\numberwithin{equation}{section}

\usepackage[T1]{fontenc}
\usepackage[latin1]{inputenc}
\usepackage{color}
\usepackage{amstext}
\usepackage{amsthm}
\usepackage{amssymb}

\makeatletter

\numberwithin{equation}{section}
\numberwithin{figure}{section}
\theoremstyle{plain}
\newtheorem{thm}{\protect\theoremname}
\theoremstyle{definition}
\newtheorem{defn}[thm]{\protect\definitionname}
\theoremstyle{plain}
\newtheorem{prop}[thm]{\protect\propositionname}
\theoremstyle{plain}
\newtheorem{cor}[thm]{\protect\corollaryname}
\theoremstyle{remark}

\theoremstyle{plain}

\usepackage{tikz}
\usepackage{pgfplots}

\makeatother

\makeatother

\providecommand{\corollaryname}{Corollary}
\providecommand{\definitionname}{Definition}
\providecommand{\lemmaname}{Lemma}
\providecommand{\propositionname}{Proposition}
\providecommand{\remarkname}{Remark}
\providecommand{\theoremname}{Theorem}

\include{amssymb}

\newcommand{\md}{{\mod{}}}

\newcommand{\ii}{{\mathrm i}}
\newcommand{\ee}{{\mathrm e}}
\newcommand{\ZZ}{{\mathbb Z}}
\newcommand{\CC}{{\mathbb C}}
\newcommand{\PP}{{\mathbb P}}
\newcommand{\pp}{{\mathcal P}}
\newcommand{\HH}{{\mathcal H}}

\begin{document}
\global\long\def\AA{\mathbb{A}}%
\global\long\def\CC{\mathbb{C}}%
 
\global\long\def\BB{\mathbb{B}}%
 
\global\long\def\PP{\mathbb{P}}%
 
\global\long\def\QQ{\mathbb{Q}}%
 
\global\long\def\RR{\mathbb{R}}%
 
\global\long\def\FF{\mathbb{F}}%

\global\long\def\DD{\mathbb{D}}%
 
\global\long\def\NN{\mathbb{N}}%
\global\long\def\ZZ{\mathbb{Z}}%
 
\global\long\def\HH{\mathbb{H}}%

\global\long\def\Gal{{\rm Gal}}%
\global\long\def\GG{\mathbb{G}}%
 
\global\long\def\UU{\mathbb{U}}%

\global\long\def\bA{\mathbf{A}}%

\global\long\def\kP{\mathfrak{P}}%
 
\global\long\def\kQ{\mathfrak{q}}%
 
\global\long\def\ka{\mathfrak{a}}%
\global\long\def\kP{\mathfrak{p}}%
\global\long\def\kn{\mathfrak{n}}%
\global\long\def\km{\mathfrak{m}}%

\global\long\def\cA{\mathfrak{\mathcal{A}}}%
\global\long\def\cB{\mathfrak{\mathcal{B}}}%
\global\long\def\cC{\mathfrak{\mathcal{C}}}%
\global\long\def\cD{\mathcal{D}}%
\global\long\def\cH{\mathcal{H}}%
\global\long\def\cK{\mathcal{K}}%

\global\long\def\cF{\mathcal{F}}%
 
\global\long\def\cI{\mathfrak{\mathcal{I}}}%
\global\long\def\cJ{\mathcal{J}}%

\global\long\def\cL{\mathcal{L}}%
\global\long\def\cM{\mathcal{M}}%
\global\long\def\cN{\mathcal{N}}%
\global\long\def\cO{\mathcal{O}}%
\global\long\def\cP{\mathcal{P}}%
\global\long\def\cQ{\mathcal{Q}}%

\global\long\def\cR{\mathcal{R}}%
\global\long\def\cS{\mathcal{S}}%
\global\long\def\cT{\mathcal{T}}%
\global\long\def\cV{\mathcal{V}}%
\global\long\def\cW{\mathcal{W}}%

\global\long\def\kBS{\mathfrak{B}_{6}}%
\global\long\def\kR{\mathfrak{R}}%
\global\long\def\kU{\mathfrak{U}}%
\global\long\def\kUn{\mathfrak{U}_{9}}%
\global\long\def\ksU{\mathfrak{U}_{7}}%

\global\long\def\a{\alpha}%
 
\global\long\def\b{\beta}%
 
\global\long\def\d{\delta}%
 
\global\long\def\D{\Delta}%
 
\global\long\def\L{\Lambda}%
 
\global\long\def\g{\gamma}%
\global\long\def\om{\omega}%

\global\long\def\G{\Gamma}%
 
\global\long\def\d{\delta}%
 
\global\long\def\D{\Delta}%
 
\global\long\def\e{\varepsilon}%
 
\global\long\def\k{\kappa}%
 
\global\long\def\l{\lambda}%
 
\global\long\def\m{\mu}%

\global\long\def\o{\omega}%
 
\global\long\def\p{\pi}%
 
\global\long\def\P{\Pi}%
 
\global\long\def\s{\sigma}%

\global\long\def\S{\Sigma}%
 
\global\long\def\t{\theta}%
 
\global\long\def\T{\Theta}%
 
\global\long\def\f{\varphi}%
 
\global\long\def\ze{\zeta}%

\global\long\def\deg{{\rm deg}}%
 
\global\long\def\det{{\rm det}}%

\global\long\def\Dem{Proof: }%
 
\global\long\def\ker{{\rm Ker}}%
 
\global\long\def\im{{\rm Im}}%
 
\global\long\def\rk{{\rm rk}}%
 
\global\long\def\car{{\rm car}}%
\global\long\def\fix{{\rm Fix( }}%

\global\long\def\card{{\rm Card }}%
 
\global\long\def\codim{{\rm codim}}%
 
\global\long\def\coker{{\rm Coker}}%

\global\long\def\pgcd{{\rm pgcd}}%
 
\global\long\def\ppcm{{\rm ppcm}}%
 
\global\long\def\la{\langle}%
 
\global\long\def\ra{\rangle}%

\global\long\def\Alb{{\rm Alb}}%
 
\global\long\def\Jac{{\rm Jac}}%
 
\global\long\def\Disc{{\rm Disc}}%
 
\global\long\def\Tr{{\rm Tr}}%
 
\global\long\def\Nr{{\rm Nr}}%

\global\long\def\NS{{\rm NS}}%
 
\global\long\def\Pic{{\rm Pic}}%

\global\long\def\Km{{\rm Km}}%
\global\long\def\rk{{\rm rk}}%
\global\long\def\Hom{{\rm Hom}}%
 
\global\long\def\End{{\rm End}}%
 
\global\long\def\aut{{\rm Aut}}%
 
\global\long\def\SSm{{\rm S}}%

\global\long\def\psl{{\rm PSL}}%
 
\global\long\def\cu{{\rm (-2)}}%
 
\global\long\def\mod{{\rm \,mod\,}}%
 
\global\long\def\cros{{\rm Cross}}%
 
\global\long\def\nt{z_{o}}%

\global\long\def\co{\mathfrak{\mathcal{C}}_{0}}%

\global\long\def\ldt{\Lambda_{\{2\},\{3\}}}%
 
\global\long\def\ltd{\Lambda_{\{3\},\{2\}}}%
\global\long\def\lldt{\lambda_{\{2\},\{3\}}}%

\global\long\def\ldq{\Lambda_{\{2\},\{4\}}}%
 
\global\long\def\lldq{\lambda_{\{2\},\{4\}}}%

\title[Modular curves $X_1(n)$ ]{Modular curves $X_1(n)$ as moduli spaces of point arrangements and applications}

\author{Lev Borisov, Xavier Roulleau}
\begin{abstract}
For a complex elliptic curve $E$ and a point $p$ of order $n$ on it, the images of the points $p_k=kp$  under the Weierstrass embedding of $E$ into $\CC\PP^2$ are collinear if and only if the sum of indices is divisible by $n$. Thus, it provides a realization of a certain matroid. We study this matroid in detail and prove that its realization space is isomorphic (over $\CC$) to the modular curve $X_1(n)$, provided $n\geq 10$, which also provides an integral model of $X_1(n)$. In the process, we find a connection to the classical Ceva and B\"or\"oczky examples of special point and line configurations. We also discuss the situation for smaller values of $n$.
\end{abstract}

\maketitle

\section{Introduction}

For $n>4$ coprime to the characteristic of the base field, modular
curves $X_{1}(n)$ were defined as the
compactifications of the moduli spaces of pairs $(E,t)$, 
where $t$ is a point of order $n$ on an elliptic curve
$E$. The first
constructions of these curves were obtained over $\CC$ 
 by taking the compactification of the quotient
$\cH/\G_{1}(n)$ of the upper half-plane
by the modular group $\G_{1}(n)$. 

\smallskip
Deligne and Rapoport in \cite{Deligne-Rapoport} constructed these
curves as representation of functors, using Deligne-Mumford stacks.
Another approach was done by Katz-Mazur in \cite{Katz-Mazur} using
Drinfeld modules. These two works have been unified by Conrad in \cite{Conrad}
using Artin stacks. The main difficulties of the constructions of
$X_{1}(n)$ are to find a model over the integers, to have
a geometric interpretation of the points at the compactification (the
cusps), and a geometric moduli interpretation of their reduction modulo
$p$ when $p$ divides $n$. 

\smallskip
We propose in this paper to give a geometric construction of the modular
curves as (compactifications of) moduli spaces of realizations of certain matroids, 
i.e. as moduli spaces of line or points arrangements. Here a line
(respectively a point) arrangement in the projective plane is a finite
set of lines (respectively points). Such arrangement is said labeled
if there is a set parametrizing its elements. 

\smallskip
Let us recall that a matroid $\cT$ (of rank $3$) is a combinatorial
structure which formally describes the intersection relations
between the elements of a labeled line arrangement, or in the dual
setting, $\cT$ encodes the alignments between the points of a labeled
point arrangement. The word ``formally'' here means that that given
a matroid such a line (or point) arrangement does not always exist,
and a matroid $\cT$ (of rank $3$) is said realizable over a field
$K$ if there is indeed a labeled line (or point) arrangement over
$K$ which has the same combinatorial data as $\cT$. There is a
well-defined notion of moduli space of realizations of a matroid,
which moduli space is a scheme naturally defined over $\ZZ$, the
realizations over the field $K$ being obtained by taking the fiber
over $\mathrm{Spec}(K)$. 

\smallskip
To a pair $(E,t)$ of an elliptic curve and an element
$t$ of order $n$, let us associate the labeled point arrangement
\[\cP_{E,t}=(kt)_{k\in\ZZ/n\ZZ},
\]
where we fix an embedding $\iota:E\hookrightarrow \PP^2$ sending 
 the group identity to a flex point on the image and where we identify a point with its image.
 Up to projective transformations, $\cP_{E,t}$ does not 
 depend on the choice of $\iota$.
Three distinct points $s_{1}=it,s_{2}=jt,s_{3}=kt$ ($i,j,k\in\ZZ/n\ZZ$)
on $\cP_{E,t}$ lie on a line if and only if $s_{1}+s_{2}+s_{3}=0$
on the elliptic curve $E$, thus if and only if $i+j+k=0$. Accordingly,
the set of triples $\{i,j,k\}\subset\ZZ/n\ZZ$ such that $i+j+k=0$
defines the so-called non-bases of a matroid $\cT_{n}$, and the
point arrangement $\cP_{E,t}$ is a realization of $\cT_{n}$. Let
$\cR(\cT_{n})$ be the moduli space of realizations of $\cT_{n}$:
this is the quotient of the space of all realization $\cP=(p_{k})_{k\in\ZZ/n\ZZ}$
of $\cT_{n}$ by the action of the projective transformations $\g$
such that $\g.\cP=(\g p_{k})_{k\in\ZZ/n\ZZ}$. We denote the class
of $\cP$ in $\cR(\cT_{n})$ by $[\cP]$, so that $[\cP]=\{\g.\cP\,|\,\g\in PGL_{2}\}$.
We obtain:
\begin{thm}
\label{thm:MAIN-Y}Suppose that the base field is an algebraically
closed field of characteristic not dividing $n$. For $n\geq7$, the
map $\phi:X_{1}(n)^\circ\to\cR(\cT_{n})$, $(E,t)\to[\cP_{E,t}]$
is birational, where $X_{1}(n)^\circ$ is the complement of the cusps in $X_{1}(n)$.
\end{thm}

For $n\geq7$, there exists a natural compactification $\overline{\cR(\cT_{n})}$
of $\cR(\cT_{n})$ and the map $\phi$ extends to a morphism 
$$\phi:X_{1}(n)\to\overline{\cR(\cT_{n})}.$$
Let us suppose that the base field is $\CC$, and let 
$$\pi:\cH^{*}\to X_{1}(n)=\cH^{*}/\G_{1}(n),\,\tau\mapsto\G_{1}(n)\tau$$
be the  quotient map.

\begin{thm}
\label{thm:MAIN-2} Let be $n\geq 10$. The map $\phi$ is an isomorphism
and the curve $\overline{\cR(\cT_{n})}$ is contained in a projective
space $\PP$ of dimension $\left\lfloor \tfrac{n-3}{2}\right\rfloor $.
The composite map $\cH^{*}\to X_{1}(n)\to\PP$ is obtained by weight
one modular forms for the group $\G_{1}(n)$.
\end{thm}

\smallskip
Borisov, Gunnels and Popescu proved in \cite{BG}, \cite{BG2} that certain weight $1$ 
Eisenstein series, obtained as log-derivatives of the Jacobi theta function at 
rational points, generate the ring of modular forms for $\Gamma_1(n)$ 
in sufficiently high degree 
and thus provide an embedding $X_1(n)\to \PP$.
When $n\geq 5$ is prime, one can even find explicit relations 
that cut out $X_1(p)$ scheme-theoretically in $\mathbb{P}^{\tfrac{p-3}{2}}$, see \cite{BGP}.

 In \cite{Voight}, Voight and Zureick-Brown study a generalization of
Petri's Theorem about the quadratic relations defining the compactification
of a curve $\cH/\G$, where $\cH$ is the upper-half plane and $\G$
is a modular group. In our situation, for a rank $3$ matroid $M$,
the ideal of the moduli space $\cR(M)$ of realizations of $M$ is
always generated by quadrics or linear forms with integral coefficients. 


\smallskip
Let us describe the structure of the paper. 
In Section \ref{sec:Action-of-},
we recall the basics of line arrangements and the notations of some operators acting 
on them defined in \cite{OSO}, which will be used later. 
We also recall the theory of matroids $\mathcal{T}$ and their
moduli space of realizations $\mathcal{R}(\mathcal{T})$. 
We then introduce the matroid $\mathcal{T}_n$ and the map $X_1(n)^\circ\to \mathcal{R}(\mathcal{T}_n)$, 
where $X_1(n)^\circ$ is the open sub-scheme parametrizing pairs $(E,t)$, where $E$ is
 an elliptic curve and $t$ is a $n$-torsion point, the complement of the cusps points in $X_1(n)$.

\smallskip
In Section \ref{sec:realiz}, we prove that $\mathcal{R}(\mathcal{T}_n)$ is one-dimensional and its generic point 
is the image of the map $\phi$.  
In Section \ref{sec:stu}, we prove several results on modular forms which are needed in the next section. In particular, we find explicit Eisenstein series for $\Gamma_1(n)$ which have zeros only at the cusps of $X_1(n)$. Their ratios could be used to study the cuspidal subgroup of the Jacobian of $X_1(n)$ (the finite subgroup generated by the differences of the cusps).
Section \ref{RelModFor} contains the proof of Theorem \ref{thm:MAIN-2}. 
Section \ref{Compact} studies two natural compactifications of $\mathcal{R}(\mathcal{T}_n)$.


\smallskip
In the ancillary files of the first arXiv version of this paper, one
can find a) the Magma algorithms used for computing the moduli space of realizations
of a matroid, and the equations of the curves $X_{1}(n)$ for low $n$, b) the Mathematica
 computations used in Sections 3, 5 and 6.

\smallskip
Finally let us observe that a similar construction may be obtained for the modular curve $X(n)$.
Also, using embedding of elliptic curves in $\mathbb{P}^{k-1}$ as degree $k$ curves,
 one may study 
the modular curves $X_1(n)$ as realization spaces of rank $k\geq 4$ matroids. 
That will be the subject of another paper.


\smallskip
\textbf{Acknowledgements.} The authors thank Lukas Kühne for useful
discussions on matroids. We used the computer algebra systems OSCAR
\cite{Oscar}, Magma \cite{Magma} and Mathematica \cite{Mathematica}. 
X.R. acknowledges support from Centre Henri Lebesgue
ANR-11-LABX-0020-01.

\section{\label{sec:Action-of-}Preliminaries on matroids, operators and cubic
curves}

\subsection{\label{subsec:Line-arrangement-and}Point and line arrangements}
A line arrangement $\cC=\ell_{1}+\dots+\ell_{n}$ is the union of
a finite number of distinct lines in $\PP^{2}$. A labeled line arrangement
$\cC=(\ell_{1},\dots,\ell_{n})$ is a line arrangement with a choice of ordering  of the lines. 

\smallskip
We denote by $\cD$ the duality operator
 which to a line arrangement $\cC$ associates the
set of the annihilators of the lines of $\cC$ in the dual projective plane $\check{\PP}^{2}$. Concretely, we fix coordinates
$x,y,z$ and then $\cD(\ell)=(a:b:c)$ for the line $\ell:\{ax+by+cz=0\}$, moreover $\cD(a:b:c)=\ell$. By duality, one also has a natural notion
of (labeled) point arrangements.

\smallskip
Recall that for a given line arrangement $\cC$ in $\PP^{2}$ and $m\geq 2$, an $m$-point of   $\cC$ is a point at which exactly $m$ lines
of $\cC$ meet. For $\mathfrak{m}$ a set of integers $m\geq 2$, 
let us denote by $\cP_{\mathfrak{m}}(\cC)$ the  union of the $m$-points
of $\cC$, over $m\in \mathfrak{m}$.  
For a given point arrangement $\cP$ and $n\geq 2$, we define 
$\cL_{\mathfrak{n}}(\cP)$
as the union of the lines that are $n$-rich, where an $n$-rich line is a line containing exactly $n$ points of
$\cP$. Of course, $\cP_{\mathfrak{m}}(\cC)$ and $\cL_{\mathfrak{n}}(\cP)$ could be empty. 
We denote by $\cD_{\mathfrak{n}}(\cL)$ the line arrangement $\cD(\cP_{\mathfrak{n}}(\cL))$ in the dual plane.

\smallskip
For $n\geq 2$, by abuse of notations, we denote by $\cL_{n}(\cP)$ the line arrangement $\cL_{\{\geq n\}}(\cP)$, and use a similar notation $\cP_{n}(\cC)$ or $\cD_{n}(\cC)$ for a line arrangement $\cC$.

\subsection{\label{sec:Matroids}Matroids}
Let us recall the following first definitions and results on matroids:

\smallskip
A matroid $M$ is a pair $M=(E,\cB)$, where $E$ is a finite set
of elements called atoms and the elements of $\cB$ are subsets of
$E$, called \textit{basis} (plural:  \textit{bases}),  subject to the following properties:\\
$\bullet$ $\cB$ is non-empty;\\
$\bullet$ if $A$ and $B$ are distinct members of $\cB$ and $a\in A\setminus B$,
then there exists $b\in B\setminus A$ such that $(A\setminus\{a\})\cup\{b\}\in\cB$. 

\smallskip
The bases have the same order $k$, called the rank of $(E,\cB)$.
An order $k$ subset of $E$ that is not a basis is called a \textit{non-basis}.


\smallskip
As we will be only concerned with line or point arrangements in
$\PP^{2}$, we now suppose that the rank of $M$ is $3$. If $m$
is the order of $E$, we may identify $E$ with the set $\{1,\dots,m\}$.

\smallskip
Matroids originated from the following kind of examples: If $\cC=(\ell_{1},\dots,\ell_{m})$
is a labeled line arrangement, the subsets $\{i,j,k\}\subset\{1,\dots,m\}$
such that the lines $\ell_{i},\ell_{j},\ell_{k}$ meet in three distinct
points are the bases of a matroid $M(\cC)$ over the set $\{1,\dots,m\}$.
We say that $M(\cC)$ is the matroid associated to $\cC$.  In the
dual setting, if $\cP=(p_{1},\dots,p_{m})$ is a point arrangement,
the order $3$ subsets $\{i,j,k\}\subset\{1,\dots,m\}$ such that
$p_{i},p_{j},p_{k}$ are not collinear are the bases of a matroid $M(\cP)$.
One has $M(\cD(\cP))=M(\cP)$. 

\smallskip
A realization (over some field, and when that exists) of a matroid
$M=(E,\cB)$ is a converse operation to the application $\cL\to M(\cL)$
or $\cP\to M(\cP)$: it is the data of a size $3\times m$ matrix
with non-zero columns $c_{1},\dots,c_{m}$, which are considered up
to a multiplication by a scalar (thus as points in the projective
plane, which points must be distinct), and such that any order $3$ subset $\{i_{1},i_{2},i_{3}\}$
of $E$ is a non-basis if and only if the size $3$ minor $|c_{i_{1}},c_{i_{2}},c_{i_{3}}|$
is zero. Then we denote by $\ell_{i}$ the line with normal vector
the point $c_{i}$. The labeled line arrangement $\cC=(\ell_{1},\dots,\ell_{m})$
is such that three lines $\ell_{i_{1}},\ell_{i_{2}},\ell_{i_{3}}$
meet at a common point if and only if $\{i_{1},i_{2},i_{3}\}$ is
a non-basis; it is a realization of $M(\cL)$. Dually, $\cP=(c_{i})_{1\leq i\leq m}$
is a point realization of $M(\cP)$.

\smallskip
The space of realizations $\kU(M)$ of a matroid $M$ is therefore
defined as follows: For $1\leq j\leq m$, let $x_{1,j},x_{2,j},x_{3,j}$
be the coordinates of the $j^{th}$ term in the product $(\PP^{2})^{m}$.
Consider the matrix $A=\left(x_{i,j}\right)_{1\leq i\leq3,1\leq j\leq m}$.
For a set $\mu=\{i,j,k\}$ of three elements in $E$, let $A_{\mu}$
be the $3\times3$ sub-matrix of $A$ obtained by taking columns $i,j,k$
and let $V_{\mu}$ be the sub-scheme of $(\PP^{2})^{m}$ defined by
$\det(A_{\mu})=0$. Consider 
\[
V(M)=\cap_{\mu\in NB}V_{\mu},
\]
where the product runs over the non-bases of $M$ and 
\[
W(M)=\cup_{\mu\in B}V_{\mu},
\]
where the union runs over the bases of $M$. The space of realizations
$\kU(M)$ of $M$ is the scheme 
\[
\kU(M)=V(M)\setminus W(M).
\]
We observe that this scheme $\kU(M)$ is naturally defined over $\ZZ$
as a sub-scheme of $(\PP_{\ZZ}^{2})^{m}$. 

\smallskip
Suppose that the realization space $\kU_{/k}$ is non-empty over some
field $k$. For a realization $\cP=(p_{1},\dots,p_{m})\in\kU(k)$
of $M$ and $\g\in PGL_{3}(k)$, the point arrangement $(\g p_{1},\dots,\g p_{m})$
is another realization. Let us denote by $[\cP]$ the orbit of $\cP$
under that action of $PGL_{3}$. The moduli space $\cR(M)$ of realizations
of $M$ is the quotient of $\kU$ by $PGL_{3}(k)$; we denote by 
\[
q:\kU(M)\to\cR(M),\,\cP\to[\cP]
\]
the quotient map. Suppose that there exist distinct elements $j_{1},\dots,j_{4}\in\{1,\dots,m\}$
such that every order $3$ subset in $j_{1},\dots,j_{4}$ is a basis
of the matroid $M$ (that will always be the case in this paper).
Then any class $[\cC]$ in $\cR(M)$ has a unique representative $\cC$
such that the points number $j_{1},\dots,j_{4}$ of $\cP$ are the
canonical basis
\begin{equation}\label{canbasis}
(1:0:0),\,(0:1:0),\,(0:0:1),\,(1:1:1).
\end{equation}
That shows that the quotient map has a section 
\[
s=s_{j_{1},\dots,j_{4}}:\cR(M)\to\kU(M),
\]
sending a class to the unique representative such that the vectors
number $j_{1},\dots,j_{4}$ are the canonical basis. Then one can
identify, although non-canonically, $\cR(M)$ with its image by $s$,
so that any element of $\cR(M)$ is not only a class but an actual realization. Let us fix such such a section: then $\cR(M)$ is the intersection
of $\kU(M)$ with the scheme defined by the ideal of the canonical
basis at the points number $j_{1},\dots,j_{4}$. We note that $s(\cR(M))$
then gets a scheme structure, defined over $\ZZ$. Let $j_{1}',\dots,j_{4}'\in\{1,\dots,m\}$
be four other elements such that such that every order $3$ subset
in $j_{1}',\dots,j_{4}'$ is a basis of the matroid $M$. Define $s_{1}=s_{j_{1},\dots,j_{4}},\,\,s_{2}=s_{j_{1}',\dots,j_{4}'}$.
There exists a unique $3\times3$ matrix $P$ with coprime coefficients
in 
\[
\ZZ[(x_{i,j})\,|\,1\leq i\leq3,\,j\in\{j_{1}',\dots,j_{4}'\}]
\]
sending the columns $c_{j_{1}'},\dots,c_{j_{4}'}$ to the canonical
basis, $P$ defines an isomorphism between $s_{1}(\cR(M))$ and $s_{2}(\cR(M))$,
thus $\cR(M)$ is well-defined over $\ZZ$.
\begin{defn}
Once a section $s:\cR(M)\to\kU(M)$ is fixed, we call the map $s\circ q$
the period map of $M$.
\end{defn}

Very often the first four elements of the matroid are such that every
order $3$ subset is a basis of the matroid $M$, so that, when speaking
of a period map, we implicitly refer to the section $s$ associated
to these four elements. We then moreover identify $\cR(M)$ with its
image by $s$.


\subsection{The matroid $\protect\cT_{n}$.}
For $n\geq3$, let $\cT_{n}$ be the matroid such that the atoms of
$\cT_{n}$ are the elements of $\ZZ/n\ZZ$ and the non-bases are the
order three subsets $\{i,j,k\}$ of $\ZZ/n\ZZ$ such that 
\[
i+j+k=0.
\]
The automorphism group of $\cT_{n}$ contains $(\ZZ/n\ZZ)^{*}$ since
the multiplication by an element of $(\ZZ/n\ZZ)^{*}$ preserves the
set of non-bases. 

\subsection{Cyclic torsion groups on elliptic curves and realizations of $\protect\cT_{n}$}
Let $E$ be an elliptic curve with neutral element $O$, and let $T$
be a cyclic torsion subgroup of order $n$. Fix a plane model of $E$
so that $O$ is a flex point. For $t,t',t''$, three distinct elements
of $T$, the points $t,t',t''$ are aligned if and only if 
\[
t+t'+t''=O.
\]
By taking a generator $t\in T$, one gets a labeling $\ZZ/n\ZZ\to T,\,k\to kt$ of $T$.
\begin{prop}
\label{prop:Real-Tn}The labeled point arrangement $\cP=(kt)_{k\in\ZZ/n\ZZ}$
associated to $(E,t)$ is a realization of the matroid $\cT_{n}$.
Up to projective automorphism, the point arrangement $\cP$ is independent
of the choice of the plane model of $E$ so that $O$ is a flex point.
\end{prop}

\begin{proof}
The first assertion is clear. The second assertion is a direct consequence
of \cite[III, Proposition 3.1]{Silverman}.
\end{proof}
As we will see in the next sections, conversely, for $n\geq7$, any
realization of $\cT_{n}$ comes from such torsion points on a (possibly
singular) cubic curve. Moreover we will obtain that, knowing $\mathcal{P}$, one may recover -up
to isomorphism- the cubic curve and the generator of the torsion group.
 
\smallskip
Let $X_{1}(n)^{\circ}\subset X_{1}(n)$ be the open subscheme of the
modular curve which parametrizes pairs $(E,t)$ of elliptic curves with
a torsion point of order $n$. Proposition \ref{prop:Real-Tn}  may be rephrased as:
\begin{cor}
There exists a morphism
\[
\phi_{n}:X_{1}(n)^{\circ}\to\cR(\cT_{n})
\]
which to a point $(E,t)$ associates the realization $(kt)_{k\in\ZZ/n\ZZ}$
up to projective automorphism.
\end{cor}

\section{Realization spaces of $\cT_n$.}\label{sec:realiz}
The aim of this section is to further study realization spaces of $\cT_n$ for $n\geq 5$. We will also be interested in their compactifications in $(\PP^2)^{n-4}$. For simplicity, we will work over the field of complex numbers. 

\smallskip
$\bullet$ Case $n=5$. We have points $p_0,\ldots,p_4$ with the non-bases $\{0,1,4\}, \{0,2,3\}$ only. This means that we can make $p_1,\ldots,p_4$ to be canonical basis \eqref{canbasis}, which then forces $p_0=(0,1,1)$, so the realization scheme is a point and $\phi_5$ is constant.

\begin{remark}It is possible to modify the matroid slightly to get an open subset of the modular curve $X_1(5)$. Namely,
consider the pair $(E,\a)$, where $E$ is an elliptic curve and $\a$
is a $5$-torsion element of  $E$. Up to projective transformation,
one may suppose that $E=E_{\t}$ is the elliptic curve 
\[
E_{\t}:\,\,y^{2}+(1-\t)xy-\t y=y^{3}-\t x^{2},\,\,(\t\in\PP^{1}\,\text{generic})
\]
and $\a=\a_{\t}=(0:0:1)$; $E_{\t}$ is the universal elliptic curve
with a point of $5$-torsion, see e.g. \cite{Baaziz}. 
Let us associate to $(E_{\t},\a)$ the labeled line arrangement $\cC=\cC_{\t}$
which is the union of the $5$ tangent lines to the points in $T=(k\a)_{k\in\ZZ/5\ZZ}$
and the two lines in $\cL_{2}(T)$ which are not the tangent lines,  
where $\cL_{2}(T)$ is the union of the lines that contain at least two points of $T$ (see Section \ref{subsec:Line-arrangement-and}).
Let $\cT_{5}'$ be the matroid with $7$ atoms
defined by the non-bases
\[
\{1,6,7\},\{2,3,7\},\{2,4,6\},\{3,5,6\},\{4,5,7\}.
\]
The realization space $\cR(\cT_{5}')$ is smooth, 1-dimensional; the seven
lines associated to a point in $\cR(\cT_{5}')$ have duals
the canonical basis of  $\PP^{2}$ and 
\[
(\t+1:1:1-\t^{2}),(\t+1:1:\t+1),(0:1:\t+1),
\]
for $\t\in\PP^{1},$ $t\notin\{0,\pm1\}$.

The line arrangement $\cC_{\t}$ is a realization of $\cT_{5}'$ and
a direct computation shows that the rational map $(E_{\t},\a_{\t})\to\cC_{\t}$
from $X_{1}(5)\simeq\PP^{1}$ to $\cR(\cT_{5}')$ has degree one. 
\end{remark}


$\bullet$ Case $n=6$. The matroid $\cT_6$ has non-bases $\{0,1,5\}$, $\{0,2,4\}$, $\{1,2,3\}$ and $\{3,4,5\}$. We can put the points $p_0,p_1,p_3,p_4$ in the standard form \eqref{canbasis}, which forces 
$$
p_2=(0,1,1),~p_5=(1,1,0).
$$
So  the realization scheme is again a point.

\begin{remark} As in the $n=5$ case, we can find a matroid description for $X_1(6)$.
The universal elliptic curve over $X_{1}(6)\simeq\PP^{1}$ is given
by 
\[
E_{\t}:\,\,y^{2}+(1-\t)xy-(\t^{2}+\t)y=x^{3}-(\t^{2}+\t)x^{2},\,\,\t\in\PP^{1}\text{ generic},
\]
(see e.g. \cite{Baaziz}, \cite{TY}). The point $\a=(0:0:1)$ has order $6$.
For $k\in\ZZ/6\ZZ$, let us denote by $p_{k\a}'$ the dual of the line tangent
to $E_{\t}$ at the point $k\a$; this is a point in the dual plane. 
The line arrangement $\cC_{\t}=\cL_{\{2\}}((p_{k\a}')_{k\in\ZZ/6\ZZ})$ (see notations in Section  \ref{subsec:Line-arrangement-and})
is the union of $15$ lines. We label these lines from $1$ to $15$
so that they correspond bijectively with the $15$ pairs $(0,1),(0,2),\dots,(4,5)$:
by that bijection the line corresponding to the pair $(i,j)$ is the
line through the points $p_{i\a}'$ and $p_{j\a}'$. 

Inspired by this, let $\cT_{6}'$ be the matroid with atoms
$\{1,\dots,15\}$ and non-bases all the order three subsets of the
following sets 
\[
\begin{array}{c}
\{1,12,13\},\{3,6,15\},\{3,8,12\},\{5,8,10\},\{1,2,3,4,5\},\{1,6,7,8,9\},\\
\{2,6,10,11,12\},\{3,7,10,13,14\},\{4,8,11,13,15\},\{5,9,12,14,15\}.
\end{array}
\]
One computes that the moduli space $\cR(\cT_{6}')$ of realizations
is $1$-dimensional. Corresponding to the above description of the
non-bases, a realization $\cC$ is a union of $15$ lines
which has four triple points and six $5$-points; moreover $\cC$
has $33$ double points. The six $5$-points of a realization $\cC_{\t}$
in $\cR(\cT_{6}')$ are 
\[
\cP_{\t}:\,\,(0:0:1),(0:1:0),(-1:1:1),(\t:0:1),(-\t:1:0),(-\t-1:1:1),
\]
for $\t\in\PP^{1}$ not in the zero loci of the polynomials
\[
x+1,\,x,\,2x+1,\,x-1,\,x^{2}+x+1,\,x^{2}-x-1.
\]
The line arrangement $\cL_{\{2\}}(\cP_{\t})$ is the realization $\cC_{\t}$
of $\cT_{6}'$. 
Then a direct computation
shows that the map 
\[
(E_{\t},\a)\to\cC_{\t}
\]
is a map from $X_{1}(6)^\circ$ to $\cR(\cT_{6}')$ and that this map has degree
$1$ (here $X_{1}(6)^\circ$ is the complement of the cusps). The curve $\cR(\cT_{6}')$ is smooth.
\end{remark}

$\bullet$ Case $n=7$. When $n\geq 7$, we can take the points $p_0,\ldots,p_3$ to be the canonical basis.
The moduli space $\cR(\cT_{7})$ of realizations
of $\cT_{7}$ is one dimensional and irreducible. The seven points
of a realization $\cC$ of $\cT_{7}$ are the four points of the canonical
base and the points
\[
(0:1:1),\,(1:0:t),\,(t-1,t,0),
\]
for $t\notin\{0,1\}$. The universal elliptic curve with $\ZZ/7\ZZ$-torsion
is 
\[
E_{t}:\,\,y^{2}+(1+t-t^{2})yx+(t^{2}-t^{3})y=x^{3}+(t^{2}-t^{3})x,
\]
and $\a=(0:0:1)$ has order $7$ (see e.g. \cite{TY}). Using the period map, one can check
that the map $\phi:X_{1}(7)^{\circ}\to\cR(\cT_{7})$ defined by 
\[
(E_{t},\a)\to(k\a)_{k\in\ZZ/7\ZZ}
\]
has degree $1$, where $X_{1}(7)^\circ$ is the complement of the cusps. 

\smallskip
$\bullet$ Case $n=8$. The non-bases of the matroid $\cT_{8}$ are
\[
\{0,1,7\},\{0,2,6\},\{0,3,5\},\{1,2,5\},\{1,3,4\},\{3,6,7\},\{4,5,7\}.
\]
One computes that the moduli space $\cR(\cT_{8})$ is one-dimensional
and irreducible; the eight points of a realization $\cP=(p_{i})_{i\in\ZZ/8\ZZ}$
of $\cT_{8}$ are 
\[
\begin{array}{c}(1:0:0),(0:1:0),(1:1:1),(0:0:1),\\[.3em]
(0:t:-1),(1:0:1),(1:t:t),(1:t:0)
\end{array}
\]
for $t\notin\{0,\pm1,\infty\}$. 
Using the universal elliptic curve with $8$-torsion elements (see \cite{Baaziz}), one can check
that the map $(E,\a) \in X_{1}(8)^\circ \to(k\a)_{k\in\ZZ/8\ZZ}\in\cR(\cT_{8})$ is well-defined and has degree one onto its image, 
 where $X_{1}(8)^\circ$ is the complement of the cusps. 



\smallskip
$\bullet$ Case $n=9$. Up to a projective transformation, the $9$ points of a realization of $\cT_9$ are 
$$
\begin{array}{c}
 (1:0:0), (0:1:0), (1:1:1), (0:0:1), (t:t:t-1), \\[.3em]
 (0:t:t-1), 
(1:0:1), (1:t:t), (1:t:0), 
\end{array}
$$
 for
 $t \notin \{0,1,\tfrac{1\pm \sqrt{-3}}{2}, \infty\}$.
 As for $n=7,8$, one may use the universal elliptic
curve to check that the map $\phi:X_{1}(9)^\circ\to\cR(\cT_{9})$ has
degree $1$. 

\smallskip
$\bullet$ Case $n\geq 10$. We will now study $\protect\cR(\protect\cT_{n})$ for $n\protect\geq10$.  
Here is an outline of our approach.

\smallskip
We begin by defining a sequence $(p_{k})_{k\in\ZZ}$ 
of points in $\PP^{2}$ such that $p_{0},p_{1},p_{2}$ are not on the same line,
 and for all sets $\{i,j,k\}\subset\ZZ$
of three distinct integers such that $i+j+k=0$, the points $p_{i},p_{j},p_{k}$
are on a line. 
We focus on which conditions that sequence becomes exactly $n$-periodic, 
so that this gives a realization of $\mathcal{T}_n$. 
 We prove that the coordinates of the points $(p_k)_{k\in\ZZ}$ 
are  parametrized by fractions of some variables $s,t$, which are subject to some relations and inequalities.
Using that $n\geq 10$, one shows that there is a unique cubic curve $E$ 
that contains the points $(p_k)_{k\in\ZZ}$. In case that cubic curve is smooth, 
it proves that the realization $(p_k)_{k\in\ZZ/n\ZZ}$ is the image by $\phi$ of a point of the modular curve, 
namely the point corresponding to the pair $(E,p_1)$. 
We then focus to the case when $E$ is singular; then $E$ is a nodal cubic curve
and the realization $(p_k)_{k\in\ZZ/n\ZZ}$ of $\mathcal{T}_n$ is a cyclic sub-group 
on the smooth part of $E$, for the natural group law with neutral element $p_0$.

\smallskip
 From the hypothesis on the sequence $(p_k)_{k\in \ZZ}$, and up
to projective transformation, one can suppose that 
\[
p_{0}=(1:0:0),\,p_{1}=(0:1:0),\,p_{2}=(0:0:1), p_3=(1:1:1).
\]
Let us suppose moreover that the sequence is $n$-periodic;  it will be convenient for us to think of the data of $p_i$ as an $n$-periodic size $\ZZ \times 3$ matrix with rows giving $p_i$, and we will now look at the consequences of the collinearity for $p_a,p_b,p_c$ for $a+b+c=0$. We will denote the three coordinates on $\PP^2$ as $(x_1:x_2:x_3)$.

\smallskip
The line $\overline{p_1p_2}=\{x_1=0\}$ contains $p_{-3}=p_{n-3}$, but since no other points are 
on this line, we may assume that $p_i = (1:y_{i,2}:y_{i,3})$ for all $i\not\in\{-3,1,2\}\md n$. Also note that since $p_{-3}\not\in \overline{p_0p_2}=\{x_2=0\}$, we have $p_{-3}=(0:1:y_{-3,3})$. 

\smallskip
The realization space of $\cT_n$ is given by periodicity and vanishing of the minors for rows $a,b,c$ with $a+b+c=0\md n$, as well as nonvanishing of the rest of the $3\times 3$ minors.
Let us denote the minor of the rows labeled $a,b,c\in\ZZ/n\ZZ$ by $\det_{a,b,c}$.

\smallskip
The point $p_{-1}$ lies on the line $\overline{p_0p_1}=\{x_3=0\}$. On the other hand, it \emph{does not} lie on $\overline{p_0p_2}=\{x_2=0\}$. Therefore,
\[
p_{-1}=(1:s:0),\, s\neq 0.
\]
To be more precise, the vanishing and nonvanishing of the appropriate determinants eliminates the entry 
$y_{-1,3}$ and forces $y_{-1,2}\neq 0$.
 Similarly, $p_{-3}$ lies on the intersection of the 
lines $\overline{p_1p_2} = \{x_1=0\}$ and $\overline{p_0p_3}=\{x_2=x_3\}$ so
\[
p_{-3}=(0:1:1).
\]
The vanishing of $\det_{-2,0,2} = -y_{-2,2}$ and $\det_{-2,-1,3} = 
s - y_{-2,2} + y_{-2,3} - s y_{-2,3}$  yields 
\[
p_{-2}=(1:0:\frac {s}{s-1})
\]
where $s-1$ is invertible since $\det_{-1,2,3} = s-1$ is nonzero. 
We then compute $\det_{-4,1,3}= 1- y_{-4,3} $ and denote $t=y_{-4,2}$, so that $$p_{-4}=(1:t:1).$$
Then $\det_{-4,0,4}=y_{4,2} - t y_{4,3}=0$ and $\det_{-3,-1,4} = -s + y_{4,2} - y_{4,3}$ allows :us to solve for 
$$
p_4= (1:\frac {st}{t-1}:\frac s{t-1})
$$
where we know that $t-1$ is invertible because it is equal to $\det_{-4,2,3}$. 
We can also compute $p_5$ by using $\det_{-3,-2,5}=\frac s{s-1} + y_{5,2} - y_{5,3}$ and $\det_{-4,-1,5}=-s + y_{5,2} + s   y_{5,3} - t   y_{5,3}$. We get 
$$
p_5=(1:\frac{(s (-1 + t))}{(s-1 ) (1 + s - t)}: \frac {s^2}{( s-1) (1 + s - t)})
$$
where $(1+s-t)$ is invertible because $\det_{-4,-3,4}=-1 - s +t$.

\smallskip
We now observe that there is a unique up to scaling cubic curve that passes through $9$ points
$p_{-3},\ldots,p_5$. (There was a dimension two space of cubics through $p_{-4},\ldots,p_4$.) We used Mathematica to compute the size $9$ minors of the matrix of the values of the $10$ cubic monomials to see that one of them equals
$$
\frac{s^8 (1 - t + s t)}{(-1 + s)^4 (1 + s - t)^2 (-1 + t)}.
$$
We know that $1-t+st$ is invertible because $\det_{-4,-3,5}=\frac {1 - t + s t}{s-1}$. This implies that the matrix of values of the $10$ cubic monomials at the above $9$ points has rank $9$. Its nullspace is one-dimensional and gives a unique cubic through the $9$ points. Explicitly, this cubic is
\[
\begin{array}{c}
E=\{F_{s,t}=0\}=\{-s^2 x_1^2 x_2 + s x_1 x_2^2 + s t x_1^2 x_3 +(s^2- s-t) x_1 x_2 x_3\\[.3em]
  + (1-s)x_2^2 x_3 + t(1-s) x_1 x_3^2  
  +(s-1)x_2 x_3^2 =0\}.
\end{array}
\]

\smallskip
If the cubic $E$ is non-singular (which is generically the case), the $n$ points $p_k=kp_1, k\in \ZZ/n\ZZ$ on $E$ form a cyclic torsion sub-group (see also Remark \ref{dim2} below).

\smallskip
Let us study when $E$ is  singular and which singularities can occur.
 First note that the partial derivatives of $F_{s,t}$ at $p_0$ are $0, -s^2, s t$ respectively, 
 which means that  $E$ is smooth at $p_0$. We can also see that the line $s x_2 = t x_3$
  is triple-tangent to $E$ at $p_0$ unless $s=t$, which is impossible because
 $\det_{-2,1,4}=\frac{s (s - t)}{(-1 + s) (-1 + t)}$.
Moreover, the partial derivative $\frac{\partial F_{s,t}}{\partial x_1}$  is 

$$
-(2 s x_1 - x_2 + (1-s) x_3 ) (s x_2 - t x_3)
$$
so singular points can occur only at $2 s x_1 - x_2 +(1-s) x_3 =0$. The restrictions of $\frac{\partial F_{s,t}}{\partial x_2}$ and $\frac{\partial F_{s,t}}{\partial x_3}$ to this line are respectively
\[
\begin{array}{c}
3 s^2 x_1^2 + (5 s - 5 s^2 - t) x_1 x_3 +(-1 + s) (-1 + 2 s) x_3^2,
\\[.5em]
-s (-2 s + 2 s^2 + t)x_1^2 + (-1 + s) (s + 3 s^2 - t) x_1 x_3 -(-1 + s)^2 (1 + s)x_3^2.
\end{array}
\]
We can take a resultant and find that the cubic $E$ is singular if and only if (using the known nonzero factors) 
\begin{align}\label{nodes}
 -9 s + 3 s^2 + 5 s^3 + s^4 + t + 10 s t - 11 s^2 t - t^2=0.
\end{align}
The singular point is given by
\[
\begin{array}{c}
(1: \frac {3 s + 4 s^2 + s^3 + t - 8 s t}{ 4 - s + s^2 -  3 t}:
\frac{ 5 s - 6 s^2 + s^3 - t + 2 s t}{(-1 + s) (4 - s + 
     s^2 - 3 t)}).
\end{array}
\]
In particular, $E$ will have at most one singular point, a node or a cusp. 

\begin{remark}\label{dim2}
Whether or not $E$ is singular, it is clear that the points $p_k$ must form an order $n$ subgroup of the group of smooth points of $E$. Their coordinates can be computed inductively as rational functions in $s$ and $t$.
(Indeed, from the property that if $i+j+k=0$ then the points $p_i,p_j,p_k$ are aligned, one may find the point $p_{k+1}$ (or $p_{-(k+1)}$) as the intersection point of two lines passing through points in the set $\{p_{-k},\dots,p_k\}$.)
Then $\cR(\cT_n)$ will be a closed subscheme of a Zariski open set in $\CC^2$.  The Zariski open set is given by the nonvanishing of the determinants, and the closed subscheme will be cut out by the $n$-periodicity conditions $p_k=p_{k+n}$.
\end{remark}

The curve \eqref{nodes} is rational, and can be parameterized by
$$
(s,t) = \big(\frac {r^2-1}{5r^2-1}, \frac {8 (r - 3 r^2 + 4 r^4)}{(5 r^2-1)^2}\big).
$$
In fact, we can also parameterize its double cover that picks the branches at the node
as 
$$
(s,t)= (\frac{(-1 + w^2) (3 + w^2)}{1 + 10 w^2 + 5 w^4 }, \frac {32 w^4 (1 + w^2) (3 + w^2)}{(1 + 10 w^2 + 5 w^4)^2}).
$$
Under this parameterization, the singularity occurs at
$$(1: \frac{4 w^2 (3 + w^2)}{1 + 10 w^2 + 5 w^4}: \frac{3 + w^2}{2 (1 + w^2)})
$$
and one may that way obtain the equation of the curve $F_{s,t}$  in terms of $w$.
We parameterize the smooth locus of the cubic curve by (generically)
$$
\begin{array}{c}
(x_1:x_2:x_3)=(1:\frac{
  4 (-1 + v) (-1 + w) w^2 (1 + w) (3 + w^2) (-1 + v - 2 w - 
     2 v w - w^2 + v w^2)}{(1 + v - w + v w) (-1 - v - 3 w + 3 v w - 
     3 w^2 - 3 v w^2 - w^3 + v w^3) (1 + 10 w^2 + 5 w^4)} \\
  :   \frac{
 (-1 + v) (-1 + w)^2 (1 + w)^2 (-1 - v - w + v w) (3 + w^2)}{
  2 (1 + w^2) (-1 + v + 2 w + 2 v w - w^2 + v w^2) (-1 - v - 3 w + 
     3 v w - 3 w^2 - 3 v w^2 - w^3 + v w^3)})
\end{array}
$$
so that the tangent lines to branches at singular points correspond to $v=0,\infty$ and  
$v=1$ gives $p_0$. Then $v=\frac{w-1}{w+1}$ gives $p_1$. The group law on the nodal cubic then implies that $p_n$ corresponds to $v=\big(\frac{w-1}{w+1}\big)^n$, so the $n$-periodicity and lack of smaller periods implies that the solutions 
correspond to 
$$
\frac{w-1}{w+1} = \zeta^a
$$
where $\gcd(a,n)=1$ and $\zeta = \ee^{2\pi\ii/n}$. We can compute 
 the points $p_k$ for $k\neq -3,1,2\md n$  to find
\begin{equation}\label{pkcusp}
\begin{array}{c}
p_k= (1:
\tfrac{ (1 - \zeta^{2 a}) (1 + \zeta^{3 a}) (1 - \zeta^{a k}) (1 - \zeta^{
     a (k + 2)})}{(1 - \zeta^{5 a}) (1 - \zeta^{a (k - 1)}) (1 - \zeta^{a (k+3)})}:
\tfrac {(1 - \zeta^a + \zeta^{2 a}) (1 - \zeta^{a k}) (1 - \zeta^{a (k + 1)})}{
(1 + \zeta^{2 a}) (1 - \zeta^{a (k - 2)}) (1 - \zeta^{
   a (k + 3)})})
     \end{array}
\end{equation}
with $p_1=(0:1:0),p_2=(0:0:1),p_{-3}=(0:1:1)$.

\begin{remark}
We will later see the above realizations of $\cT_n$ as images of certain cusps of the modular curve $X_1(n)$.
\end{remark}

Summarizing, we obtained that:

\begin{proposition}\label{set-th}
For $n\geq 10$,  the realization space  $\cR(\cT_n)$ is set-theoretically equal to the union of the image $\phi_{n}(X_{1}(n)^{\circ})$ of the  complement of $X_1(n)^\circ$ of the cusps in the modular curve $X_1(n)$ and the points in Equation \eqref{pkcusp}.
\end{proposition}

\begin{proof}
If $E$ is nonsingular, then we see that the configuration $\{p_k\}$ is in the image of  $X_1(n)^\circ$, and the singular case has been computed above.
\end{proof}

We will later prove that $\cR(\cT_n)$ is a smooth curve, but for now we will only argue that the tangent space to $\cR(\cT_n)$ at any configuration ${\bf p} = \{p_k\}$ is of dimension at most one.
\begin{proposition}\label{atmost1}
Let ${\bf p} = \{p_k\}$ be a specific realization of $\cT_n$ over $\CC$. Then the tangent space to 
 $\cR(\cT_n)$ at $\bf p$ is at most one-dimensional.
\end{proposition}

\begin{proof}
We saw in Remark \ref{dim2} that $\cR(\cT_n)$ is locally given by vanishing and non-vanishing of some polynomials in $(s,t)$. 
As above, we consider the cubic curve $E_{\bf p}$ passing through the points in ${\bf p}$. There is a group structure on the space of $\CC[\varepsilon]/(\varepsilon^2)$-points in the smooth part of $E_{\bf p}$.
Pick a nonzero deformation $p_1(\varepsilon)$ of $p_1 = (0:1:0)$ and consider the 
$\CC[\varepsilon]/(\varepsilon^2)$ matrix of size $\ZZ\times 3$ where $p_k$ is given by 
$$
p_k = k p_1(\varepsilon) \md \varepsilon^2.
$$
Note that this matrix will no longer be periodic, but it will satisfy the vanishing of $\det_{a_1,a_2,a_3}$ with $a_1+a_2+a_3=0$ modulo $\varepsilon^2$.
We can multiply it by an invertible matrix with values in  $\CC[\varepsilon]/(\varepsilon^2)$ to get
the points $p_0,\ldots,p_3$ into the standard form. This provides a tangent vector to $\bf p$ in the $\CC^2$ with coordinates $(s,t)$ that is not tangent to $\cR(\cT_n)$, and the claim follows.
\end{proof}

\begin{remark}
If we knew that the realization space is connected, the above Proposition \ref{atmost1} would imply that it is a smooth curve. We will prove it in Section \ref{RelModFor}.
\end{remark}

\begin{remark}
It is sensible to ask about the compactification of the realization space of $\cT_n$ in $(\PP^2)^{n-4}$. We will do it in the latter sections.
\end{remark}

A priori, it is not obvious that the scheme-theoretic analog of the above 
Proposition \ref{set-th} holds, although Proposition \ref{atmost1} comes pretty close. We will  prove it in Section \ref{RelModFor}.


\section{Log-derivatives of the theta function}\label{sec:stu}
In this section we recall some known results about log-derivatives of the theta functions that almost generate the ring of modular forms for the congruence subgroup $\Gamma_1(n)$.

\smallskip
We start with the Jacobi theta function $\theta:\CC\times \HH\to \CC$ defined by 
\begin{equation}\label{prodtheta}
\theta(z,\tau) = \ee^{\frac {\pi \ii\tau}4}(2\sin\pi z) \prod_{l=1}^\infty(1-\ee^{2\pi\ii l \tau})
\prod_{l=1}^\infty(1-\ee^{2\pi\ii (l \tau+z)})\prod_{l=1}^\infty(1-\ee^{2\pi\ii (l \tau-z)}).
\end{equation}
The convergence of the product is clear, as are the transformation properties
$$
\theta(-z,\tau) = -\theta(z,\tau),~\theta(z+1,\tau) =  \theta(z,\tau),~\theta(z+\tau,\tau) = -\ee^{-2\pi\ii z-\pi \ii \tau}\theta(z,\tau).
$$
It is also easy to show that for fixed $\tau$ the function $\theta(z,\tau)$ has a simple zero precisely at $z\in L = \ZZ + \ZZ\tau$.
It is notably harder to prove, but $\theta$ also satisfies a nice functional equation under the Jacobi transformations
$$
(z,\tau) \to \Big(\frac z{cz+d},\frac {a\tau+b}{c\tau+d}\Big)
$$
for $\left(\begin{array}{cc}a&b\\c&d\end{array}\right)\in {\mathrm {SL}}(2,\ZZ)$, see \cite{Chandrasekharan}.      

\smallskip
Let $n\geq 3$ be a positive integer. For integers $a$ not divisible by $n$ we introduce the functions 
$$
r_{a}(z,\tau) = 
\frac{\theta(\frac an+z,\tau)\theta_z(0,\tau)}{\theta(z,\tau)\theta(\frac an,\tau)},
$$
where $\theta_z$ is the partial derivative with respect to $z$.
We will routinely suppress their dependence on the variable $\tau$ in our notation. Properties of $\theta$ imply that 
functions $r_a(z)$  have a simple zero at $z=-\frac an \md L$, a pole of order one at $L$ and the transformation properties
$$
r_a(z+1)=r_a(z),~r_a(z+\tau) =\zeta_n^{-a} r_a(z)
$$
where $\zeta_n$ denotes $\ee^{\frac {2\pi\ii}n}$. 
The $z$-independent factors of $r_a(z)$ are picked in a way to assure that the residue of $r_a(z)$ at $z=0$ is equal to $1$, so
we have the Laurent expansion
\begin{align}\label{Laurent}
\frac 1z + s_a + t_a z + u_a z^2 + \ldots
\end{align}
where $s_a,t_a,u_a$ are functions of $\tau$. In particular, the product expansion of $\theta$ implies that 
$$
s_a(\tau) = \frac {\theta_z(\frac an,\tau)}{\theta(\frac an,\tau)}=
2\pi\ii\Big(\frac {\zeta_n^a+1}{2(\zeta_n^a-1)} -\sum_{d>0}\ee^{2\pi \ii d \tau}\sum_{k\vert d}\big(\zeta_n^{ka}-\zeta_n^{-ka}\big)\Big).
$$
The functions $s_a=s_a(\tau)$ are modular forms of weight $1$ for the congruence subgroup $\Gamma_1(n)\subset  {\mathrm SL}(2,\ZZ)$. They depend on $a\md n$ only and satisfy 
$s_{-a}=-s_a$. 
Similarly,
$t_a$ and $u_a$ are modular forms of weight $2$ and $3$ respectively, and they satisfy $t_{-a}=t_a$ and $u_{-a}=-u_a$.

\smallskip
It has been shown in \cite{BG,BG2} 
that $s_a$ \emph{almost} generate the ring of modular forms for $\Gamma_1(n)$. Namely, any modular form of weight $3$ or higher can be expressed as a polynomial in $s_a$. The situation is notably more complicated for 
weight $2$ where the cuspidal part of the span of $s_as_b$ is precisely the span of Hecke eigenforms of analytic rank zero, see \cite{vanish}. There is also a similar statement involving the Wronskians of $s_a$ and $s_b$, see \cite{wronskian}.

\smallskip
In addition to the symmetry relations, the forms $s_a$ and $t_a$ satisfy the following.
\begin{proposition}\label{st}
If $a,b,c$ are nonzero elements of $\ZZ/n\ZZ$ such that $a+b+c=0\md n$, then 
$$
s_a s_b + s_b s_c + s_c s_a +t_a+t_b+t_c = 0.
$$
\end{proposition}

\begin{proof}
Since $\zeta_n^a\zeta_n^b\zeta_n^c=1$, the product
$r_a(z)r_b(z)r_c(z)$ is elliptic with respect to $L$. Its only poles are at $z\in L$, thus the residue at $z=0$ must be zero. It remains to use the
expansion \eqref{Laurent}.
\end{proof}

\begin{remark}
For large enough $n$, these are not the only degree two relations on $s_a,t_a$ due to presence of Hecke eigenforms of positive analytic rank. However, for $n$ prime the above relations cut out $X_1(n)$ scheme-theoretically \cite{BGP}.
\end{remark}

Recall that the Weierstrass $\pp$ function defined as 
$$
\pp(z) = \frac 1{z^2} + \sum_{l\in L,l\neq 0}\Big(\frac 1{(z-l)^2}-\frac 1{l^2})
$$
and its derivative $\pp'(z)$ can be used to embed $E = \CC/L$ into $\CC\PP^2$ by
$$
z\mapsto\left\{ \begin{array}{ll}(\pp(z):\pp'(z):1),&z\not\in L,\\
(0:1:0),&z\in L.
\end{array}
\right.
$$
In what follows we find the expression for the values of the Weierstrass function and its derivative at $\frac an$  in terms of $s,t,u$. The latter will not be used in the rest of the paper but is easy to get with the same method.

\smallskip
\begin{proposition}\label{p-st}
Let $\pp(z)$ be the Weierstrass function for the lattice $L$.
Then
$$
\pp\Big(\frac an \Big) = s_a^2 -2t_a,~~\pp'\Big(\frac an\Big ) =  -2s_a^3  + 6s_a t_a- 6 u_a.
$$
\end{proposition}

\begin{proof}
Consider 
$$
r_a(z) r_{-a}(z) =  \Big(\frac 1z + s_a  + t_a z + \ldots\Big) \Big(\frac 1z - s_a  + t_a z + \ldots\Big)
= \frac 1{z^2} - s_a^2  + 2 t_a + \ldots
$$
It is an elliptic function with only pole of order $2$ at $z=0$. Moreover, since its Laurent power series has coefficient $1$ at $z^{-2}$, it is equal to 
$$
\pp(z)  - s_a^2  + 2 t_a
$$
because the Laurent expansion of $\pp(z)$ has coefficients $1$ and $0$ at $z^{-2}$ and $z^0$ respectively.
Since  $r_a(\frac an)=0$, we get $\pp(\frac an) = s_a^2 -2t_a$.

\smallskip
We can find the value of the derivative $\pp'(\frac an)$ by a similar calculation. We consider 
\begin{align*}
r'_a(z) r_{-a}(z) =  \Big(-\frac 1{z^2} + t_a + 2u_a z + \ldots\Big)\Big(\frac 1z - s_a  + t_a z - u_a z^2 + \ldots\Big)
\\
=-\frac 1{z^3}+ s_a\frac 1{z^2}  + (3 u_a-s_at_a)+ \ldots
\end{align*}
which is an elliptic function with pole of order $3$ at $z=0$ and is therefore a linear combination of $\pp'(z)$, $\pp(z)$ and $1$. Since $\pp'$ and $\pp$ do not have any constant terms in their Laurent expansions at $z=0$, we get 
\begin{align*}
r'_a(z) r_{-a}(z) = \frac 12 \pp'(z) + s_a \pp(z) + ( 3 u_a-s_a t_a ).
\end{align*}
As before, we get 
$$
 \frac 12 \pp'(\frac an ) + s_a \pp(\frac an) + (3 u_a-s_a t_a ) = 0
$$
so 
$$
\pp'(\frac an ) = -2s_a \pp(\frac an) - 6 u_a +  2s_a t_a = -2s_a^3  + 6s_a t_a- 6 u_a.
$$
\end{proof}

We will now prove several relations on $s_a$ and $t_a$ that will be used later in the paper. They are also of 
independent interest. We start with a simple formula which relates the product of two $r$ functions with another $r$ function and its derivative.
\begin{proposition}\label{rr}
For $a,k\mod n$ such that $a,k,a+k\neq 0\mod n$, we have
$$
r_a(z) r_k(z) = -r'_{a+k}(z) + (s_a+s_k) r_{a+k}(z).
$$
\end{proposition}

\begin{proof}
The product $r_a(z) r_k(z)$ is an almost periodic function for $L$, with $\ZZ/n\ZZ$ character $-(a+k)$ and double pole at $0$. The space of such functions is two-dimensional, by the Riemann-Roch formula. Both $r'_{a+k}$ and $r_{a+k}$ are in this space and they are linearly independent due to their Laurent expansions at $z=0$. Finally, we can find the precise coefficients by looking at the Laurent expansion 
$$
r_a(z) r_k(z)  = \Big(\frac 1z + s_a + \ldots\Big)\Big (\frac 1z + s_k +\ldots\Big) = \frac 1{z^2} + (s_a+s_k)\frac 1z + \ldots.
$$
\end{proof}

\begin{proposition}\label{main}
Let $a,b\md n$ be different nonzero elements of $\ZZ/n\ZZ$. Then for any $k\not\in \{a-b,0,a,-b\}\md n$ the product
$$
(s_{k+b-a}-s_k-s_b+s_a)(s_{k+b}-s_{k-a}-s_b-s_a)
$$
is the same. Specifically, it equals
$
s_b^2-s_a^2+2t_a-2t_b$.
\end{proposition}

\begin{proof}
After a rather miraculous rearrangement of the terms, the product is equal to
\begin{align*}
s_b^2-s_a^2+
\big( s_{k+b-a}(s_{k+b}-s_a)+ s_a s_{k+b} \big)  - \big(s_{k+b-a}(s_{k-a} + s_b) - s_b s_{k-a} \big)
\\
-\big(s_k(s_{k+b}-s_b)+s_b s_{k+b}\big) +\big( s_k(s_{k-a}+s_a)- s_a s_{k-a}\big) 
\\ 
 =
 s_b^2-s_a^2+
 (t_{k+b-a}+t_{k+b}+t_a)-(t_{k+b-a}+t_{k-a}+t_b) 
 - (t_k+t_{k+b}+t_{b}) 
 \\+ (t_k+t_{k-a}+t_a)
 =s_b^2-s_a^2+2t_a-2t_b.
\end{align*}
Here we used relations $s_i s_j+ s_j s_k+ s_i s_k + t_i +t_j +t_k = 0$ for $i+j+k = 0\mod n$ proved in Proposition \ref{st} and symmetry relations 
$s_{-i} = -s_i$, $t_{-j}= t_j$.
\end{proof}

\begin{corollary}\label{cusponly}
If
$a\neq \pm b \mod n$ and $k\not\in \{a-b,0,a,-b\}\md n$ 
the weight one modular form
$
s_{k+b-a}-s_k-s_b+s_a
$ 
has zeros only at the cusps of $X_1(n)$.
\end{corollary}

\begin{proof}
By Propositions \ref{main} and \ref{p-st},
$$
(s_{k+b-a}-s_k-s_b+s_a)(s_{k+b}-s_{k-a}-s_b-s_a) = \pp\Big(\frac bn\Big) - \pp\Big(\frac an\Big).
$$
It remains to observe that $\pp(\frac an) = \pp(\frac bn)$ implies that $\frac an \in\{\pm\frac bn \}\mod L$.
\end{proof}

\begin{corollary}\label{bk}
For $n\geq 7$ and $k\not\in \{-3,-1,0,2\}\md n$, there holds
$$
 \frac {s_4-2s_3+s_2}{s_{k+1}-s_k + s_2-s_3 }
=   \frac{s_{k+3}-s_{k-2}-s_3-s_2}{s_6-s_3-s_2-s_1 }$$
as an equality of modular functions for $\Gamma_1(n)$.
\end{corollary}

\begin{proof}
As a consequence of Proposition \ref{main} for $a=2$ and $b=3$ there holds
$$
(s_{k+1}-s_k-s_3+s_2)(s_{k+3}-s_{k-2}-s_3-s_2) = (s_4 - 2 s_3 + s_2)(s_6-s_3-s_2-s_1)
$$
because the right hand side is the value at $k=3$.
\end{proof}

\begin{corollary}\label{ak1}
For $n\geq 7$ and $k\not\in \{-3,-2,0,1\}\md n$, there holds
$$
 \frac {s_5-2s_3+s_1}{ s_{k+2}-s_k + s_1-s_3 } = 
 \frac{s_{k+3}-s_{k-1}-s_3-s_1}{s_6-s_3-s_2-s_1}
$$
as an equality of modular functions for $\Gamma_1(n)$.
\end{corollary}

\begin{proof}
As a consequence of Proposition \ref{main} for $a=1$ and $b=3$, for any $k\not\in \{-3,-2,0,1\}\md n$ there holds
$$
(s_{k+2}-s_k - s_3 + s_1 )(s_{k+3}-s_{k-1}-s_3-s_1) = (s_5 - 2 s_3 + s_1)(s_6-s_3-s_2-s_1)
$$
because the right hand side is the value at $k=3$.
\end{proof}

We will now also compute the values of $s_k$ at the cusps of $X_1(n)$.
\smallskip
Every cusp of $X_1(n)$ can be represented by $\frac ac\in\QQ$ and written as $\gamma (\ii\infty)$ for some $\gamma =\left(\hskip-3pt\begin{array}{cc}a&b \\ c&d\end{array}\hskip-3pt\right) \in SL(2,\ZZ)$, so we need to understand
$
\lim_{\tau \to \ii\infty} (c\tau+d)^{-1}s_k(\frac {a\tau+b}{c\tau+d}).
$
\begin{proposition}\label{cuspvalues}
If $\frac {kc}n\not\in\ZZ$, the value of $s_k$ at the cusp $\frac ac$ is
$$
\lim_{\tau \to \ii\infty} (c\tau+d)^{-1}s_k(\frac {a\tau+b}{c\tau+d})
=(2\pi \ii)(\big\{ \frac {kc}n \big\}-\frac 12 )
$$
and if $\frac {kc}n\in\ZZ$ we get 
$$
\lim_{\tau \to \ii\infty} (c\tau+d)^{-1}s_k(\frac {a\tau+b}{c\tau+d})
=(2\pi \ii)\Big (-\frac {1}{2} - \frac {\zeta_n^{kd}}{1-\zeta_n^{kd}} \Big)
=(2\pi \ii) \frac {\zeta_n^{kd}+1}{2(\zeta_n^{kd}-1)}
$$
where $\zeta_n=\ee^{2\pi\ii/n}$.
\end{proposition}

\begin{proof}
By \cite{Chandrasekharan}, we have 
\[
\theta(\frac z{c\tau+d},\frac {a\tau+b}{c\tau+d}) = \xi (c\tau+d)^{\frac 12} \ee^{\frac {\pi \ii c z^2}{c\tau+d}}
\theta(z,\tau)
\]
for some constant $\xi$, so we get
\[
\log\theta(\frac z{c\tau+d},\frac {a\tau+b}{c\tau+d}) =f(\tau)+{\frac {\pi \ii c z^2}{c\tau+d}} + \log
\theta(z,\tau).
\]
We substitute $z(c\tau+d)$ for $z$ to get 
\[
\log\theta(z,\frac {a\tau+b}{c\tau+d}) =f(\tau)+ {\pi \ii c z^2}(c\tau+d) + \log
\theta(z({c\tau+d}),\tau).
\]
Differentiation and plugging in $z=\frac kn$ yields
$$
(c\tau+d)^{-1}s_k(\frac {a\tau+b}{c\tau+d})  = 2\pi \ii c \frac kn + (\frac \partial{\partial z}
\log \theta)(\frac kn({c\tau+d}),\tau).
$$
We use the product expansion \eqref{prodtheta} and 
$\sin(\pi z) =  \frac{\ee^{\pi\ii z}-\ee^{-\pi\ii z}} {2i}
=-\frac 1{2\ii}\ee^{-\pi\ii z}(1-\ee^{2\pi\ii z})$ to get  
\[
 \big(\frac \partial{\partial z}
\log \theta)(z,\tau) = 2\pi\ii(-\frac 12 - \sum_{l=0}^\infty \frac {\ee^{2\pi\ii (l\tau + z)}}
 {1-\ee^{2\pi\ii (l\tau + z)}}
 + \sum_{l= 1}^\infty \frac {\ee^{2\pi\ii (l\tau - z)}}
 {1-\ee^{2\pi\ii (l\tau - z)}}
 \big).
\]
Therefore,
\begin{align*}
(c\tau+d)^{-1}s_k(\frac {a\tau+b}{c\tau+d})  = 2\pi \ii \big(\frac { kc}n 
-\frac 12
-
\sum_{l=0}^\infty \frac{q^{l+\frac {kc}n}\zeta_n^{kd}}{1-q^{l+\frac {kc}n}\zeta_n^{kd}}
+
\sum_{l=1}^\infty \frac{q^{l-\frac {kc}n}\zeta_n^{-kd}}{1-q^{l-\frac {kc}n}\zeta_n^{-kd}}
\big)
 \end{align*}
where $\zeta_n=\ee^{2\pi \ii/n}$ and $q^\alpha=\ee^{2\pi\ii \alpha\tau}$. 
We observe that 
\[
\lim_{\tau\to\ii\infty} \frac{q^{l\pm\frac {kc}n}\zeta_n^{\pm kd}}{1-q^{l\pm\frac {kc}n}\zeta_n^{\pm kd}}
=\left\{
\begin{array}{ll}
0, &l\pm  \frac {kc}n  >0,\\
 \frac{\zeta_n^{\pm kd}}{1-\zeta_n^{\pm kd}}
 ,&l\pm \frac {kc}n = 0,\\
-1
 ,&l\pm  \frac {kc}n  < 0,\\
\end{array}
\right.
\]
which implies the claim.
\end{proof}

\section{Realization of $\cT_n$ in terms of modular forms.}\label{RelModFor}
We will now use the results of Section \ref{sec:stu} to compute the  realizations 
explicitly in terms of modular forms on $X_1(n)$. 

\smallskip
We can view the elliptic curve $E$ as $\CC/(\ZZ + \ZZ \tau)$, with the torsion point $t=\frac 1n$.
The homogeneous coordinates of the images of the points $z = \frac in$ for $i=0 ,\ldots, n-1$ under the standard embedding are given by the $n\times 3$ matrix
\begin{align}\label{raw}
\left(
\begin{array}{ccc}
0&1&0\\[.5em]
\pp(\frac 1n) &\pp'(\frac 1n) &1 \\[.5em]
\pp(\frac 2n) &\pp'(\frac 2n) &1\\[.5em]
\pp(\frac 3n) &\pp'(\frac 3n) &1\\[.5em]
\ldots &\ldots &\ldots\\[.5em]
\pp(\frac kn) &\pp'(\frac kn) &1\\[.5em]
\ldots &\ldots &\ldots\\[.5em]
\end{array}
\right).
\end{align}
This matrix is then reduced to the standard form
by right multiplication by $\mathrm {GL}(3,\CC)$ and scaling of the rows
\begin{align}\label{std}
\left(
\begin{array}{ccc}
1& 0 & 0 \\[.5em]
0& 1 & 0 \\[.5em]
0& 0 & 1 \\[.5em]
1& 1 & 1\\[.5em]
\ldots &\ldots &\ldots\\[.5em]
1&a_k&b_k\\[.5em]
\ldots &\ldots &\ldots
\end{array}
\right)
\end{align}
with the exception of the $(n-2)$-nd row where the point lies on the first coordinate line and thus the first coordinate can not be made $1$.

\begin{remark}
The $(n-2)$-nd row of \eqref{std} that corresponds to the point $\frac {n-3}n$ is 
$
(0,1,1)
$,
because it lies on the intersection of the first coordinate line with the line on which the second and third coordinates are equal.
\end{remark}

\begin{proposition}\label{bkform}
There holds
$$
b_k = 
 \frac{s_{k+3}-s_{k-2}-s_3-s_2}{s_6-s_3-s_2-s_1 }
$$
for $k\neq 1,2,(n-3)\mod n$.
\end{proposition}

Before the proof of Proposition \ref{bkform}, let us note the following
\begin{remark}
Recall from Corollary \ref{cusponly} that the modular form $s_6-s_3-s_2-s_1$ is nonvanishing on $X_1(n)^\circ$.
\end{remark}

\begin{remark}
Note that  Proposition \ref{bkform} implies $b_{n-4}=1$.
\end{remark}

\begin{proof}
Observe that ratios of minors which have each index appearing 
equally many times in the numerator and the denominator, stay invariant under the reduction from \eqref{raw} to \eqref{std}. Therefore,
$$
\frac {\det_{01k}\det_{123}}{\det_{12k}\det_{013}}
$$
must be the same for \eqref{raw} and \eqref{std}. Here it is convenient to index rows of the matrix starting from $0$. The value of this ratio in \eqref{std} is $b_k$, so we get a formula for it. Specifically we get  $b_k=F(\frac kn)$ where
$$
F(z) = \frac {\det_{014}\det_{123}}{\det_{124}\det_{013}}
$$
of the matrix 
\begin{align}\label{3by5}
\left(
\begin{array}{ccc}
0&1&0\\[.5em]
\pp(\frac 1n) & \pp'(\frac 1n) &1\\[.5em]
\pp(\frac 2n) & \pp'(\frac 2n) &1\\[.5em]
\pp(\frac 3n) & \pp'(\frac 3n) &1\\[.5em]
\pp(z) & \pp'(z) &1
\end{array}
\right).
\end{align}

\smallskip
We will now identify the elliptic function $F$ based on the following properties which determine it uniquely.
\begin{itemize}
\item
$F(z)$ has simple poles at $z=\frac 2n$ and $z=-\frac 3n$. It has simple zeros at $z=0$ and $z=-\frac 1n$.
\item
$F(\frac 3n) = 1$.
\end{itemize}
Indeed, $\det_{014}=\pp(z)-\pp(\frac 1n)$ has simple zeros at $\frac 1n$ and $-\frac 1n$ and order two pole at $z=0$. Similarly $\det_{124}$ has simple zeros at $\frac 1n$, $\frac 2n$ and $-\frac 3n$ and order three pole at $z=0$.

\smallskip
By the first property, we can write $F(z)$ up to constant as 
$$
\frac {r_{1}(z)}{r_{-2}(z) r_{3}(z)}
$$
since it is an elliptic function with the same zeros and poles.
By Proposition \ref{rr}
$$
r_{-2}(z) r_{3}(z) 
=- r'_{1}(z) + (-s_2+s_3) r_{1}(z)
$$
so 
$$
F(z) = c \frac 1{- \frac {r'_{1}(z)}{r_{1}(z)}  + (-s_2+s_3) }
= \frac 1{ -\frac {\theta'(z+\frac 1n)}{\theta(z+\frac 1n)} +\frac {\theta'( z)}{\theta(z)}  + (-s_2+s_3) }
$$
So we get
$$
F(\frac kn) =  \frac {-c}{s_{k+1}-s_k + s_2-s_3 }.
$$
We then use the second property to conclude that 
$$
b_k = \frac {s_4-2s_3+s_2}{s_{k+1}-s_k + s_2-s_3 },
$$
and the result follows from Corollary \ref{bk}.
\end{proof}

We will now find the values of $a_k$.
\begin{proposition}\label{akform}
There holds
$$
a_k=
\frac{s_{k+3}-s_{k-1}-s_3-s_1}{s_6-s_3-s_2-s_1}
$$
for $k\neq 1,2,(n-3)$.
\end{proposition}

\begin{proof}
Let us use a similar approach to find it. We get $a_k = G(\frac kn)$ for 
$$
G(z)= \frac {\det_{024}\det_{123}}{\det_{124}\det_{023}}
$$
for the matrix \eqref{3by5}.
Then $G(\frac 3n) =1$, $G(z)$ has simple zeros at $0$ and $-\frac 2n $ and simple poles at $\frac 1n $ and $ -\frac 3n$.
We get 
$$
G(z) = c\frac {r_{2}(z)}{r_{-1}(z)r_{3}(z)}=
  \frac c{- \frac {r'_{2}(z)}{r_{2}(z)}  + (-s_1+s_3) }
= \frac c{ -\frac {\theta'(z+\frac 2n )}{\theta(z+\frac 2n )} +\frac {\theta'( z)}{\theta(z)}  + (-s_1+s_3) }
$$
and
$$
G(\frac kn) =  \frac {-c}{ s_{k+2}-s_k + s_1-s_3 }
$$
which then leads to
$$
G(\frac kn) =  \frac {s_5-2s_3+s_1}{ s_{k+2}-s_k + s_1-s_3 }
$$
and we use Corollary \ref{ak1}.
\end{proof}

\begin{proposition}\label{all}
For $n\geq 10$ and any $i\neq 0\mod n$, the 
modular function 
$$
\frac {s_i}{s_6-s_3-s_2-s_1}
$$
for $\Gamma_1(n)$ can be written as a constant linear combination of the $a_k$ and $b_k$, $k\in \{0,\dots,n-1\}$.
\end{proposition}

\begin{proof}
In view of Propositions \ref{bkform} and \ref{akform},
we simply need to show that linear combinations of
\begin{align}\label{ab}
s_{k+3}-s_{k-2}-s_3-s_2, ~s_{k+3}-s_{k-1}-s_3-s_1
\end{align}
for $k\neq 1,2,(n-3)\mod n$ 
span the space of all $s_i$. Denote by $V$ the span of the above relations. Subtracting one from the other, we get
$$
s_{k-1}=s_{k-2}  +(s_2-s_1) \mod V.
$$
for all $k \neq 1,2,(n-3)\mod n$. This implies
\begin{align}\label{modV}
s_k = s_1 + (k-1)(s_2-s_1) \mod V
\end{align}
for $k\in \{1,2,\ldots,n-5\}$.

\smallskip
Assume $n\geq 10$.
We  use \eqref{modV} and \eqref{ab} for $k=3$ 
to get
$$
s_6 -s_3 - s_1-s_2 = (5-2)(s_2-s_1) - s_1 - s_2 = 2s_2 - 4s_1 = 0 \mod V.
$$
We also use 
$$0=s_5+s_{n-5} = 2s_1 + (n-2)(s_2-s_1) = (n-2) s_2 - (n-4)s_1 \mod V.$$
The determinant 
$4(n-2)-2(n-4)=2n$
is nonzero, which means that $s_1,s_2$ and hence all $s_i$ are in $V$.
\end{proof}


\begin{definition}
Let $ Y_1(n)$ be the open part of $X_1(n)$ on which the modular form
$$
{s_6-s_3-s_2-s_1}
$$
does not vanish. Then we define the morphism
$$
\tilde\psi_n :   Y_1(n) \to M_{n\times 3}(\CC).
$$
by sending $\tau$ to the $n\times 3$ matrix  (with rows counted from $0$) 
\begin{align}\label{matrixs}
\left(
\begin{array}{ccc}
1& 0 & 0 \\[.5em]
0& 1 & 0 \\[.5em]
0& 0 & 1 \\[.5em]
1& 1 & 1\\[.5em]
\ldots &\ldots &\ldots
\\[.5em]
1& \frac{s_{k+3}-s_{k-1}-s_3-s_1}{s_6-s_3-s_2-s_1}&
\frac{s_{k+3}-s_{k-2}-s_3-s_2}{s_6-s_3-s_2-s_1}
\\[.5em]
\ldots &\ldots &\ldots
\end{array}
\right)
\end{align}
and $k=n-3$ row is given by $(0,1,1)$.
\end{definition}

\begin{proposition}\label{embedding}
For $n\geq 10$,
the map $\tilde \psi_n$ is a closed embedding of $Y_1(n)$ into $M_{n\times 3}(\CC).$
\end{proposition}

\begin{proof}
This follows from Proposition \ref{all} and the statement that any modular form for $\Gamma_1(n)$ of weight at least three can be written as a polynomial in $s_k$.
\end{proof}

We are now ready to state and prove our main theorem.
\begin{theorem}
For $n\geq 10$, the realization space $\cR(\cT_n)$ is scheme-theoretically isomorphic to the 
open subset of $X_1(n)$ obtained by adding to $X_1(n)^\circ$ (the complement of the cusps) the $(\ZZ/n\ZZ)^*$-orbit of the $\ii\infty$ cusp.
\end{theorem}

\begin{proof}
In view of Propositions \ref{bkform} and \ref{akform}, the morphism $\tilde\phi_n$ extends the map 
$\phi_n:X_1(n)^\circ\to\cR(\cT_n)$ under the identification of the latter as a closed subset of the open subset of the matrices. In particular, for any cusp in $Y_1(n)$, the matrix \eqref{matrixs} satisfies
$\det_{i_1,i_2,i_3}=0$ for $i_1+i_2+i_3=0\md n$.

\smallskip
We will now determine which of the cusps $\frac ac$ of $Y_1(n)$ map to $\cR(\cT_n)$, i.e. which cusps satisfy
$\det_{i_1,i_2,i_3}\neq 0$ for $i_1+i_2+i_3\neq 0\md n$. 

\smallskip
Consider the cusps in the $(\ZZ/n\ZZ)^*$-orbit of the $\ii\infty$ cusp. In this case $c=0\md n$ with $d$ coprime to $n$. Proposition \ref{cuspvalues} implies that the $k$-th row of the matrix \eqref{matrixs} is, for 
$\alpha = \ee^{2\pi\ii \frac dn}$,
\begin{align}
\begin{array}{l}
\Big(1,  \frac{\frac {\alpha^{k+3}+1}{\alpha^{k+3}-1} -\frac {\alpha^{k-1}+1}{\alpha^{k-1}-1}-\frac {\alpha^{3}+1}{\alpha^{3}-1}-\frac {\alpha+1}{\alpha-1}}
{\frac {\alpha^{6}+1}{\alpha^{6}-1} -\frac {\alpha^{3}+1}{\alpha^{3}-1} -\frac {\alpha^{2}+1}{\alpha^{2}-1} -\frac {\alpha+1}{\alpha-1}},
 \frac{\frac {\alpha^{k+3}+1}{\alpha^{k+3}-1} -\frac {\alpha^{k-2}+1}{\alpha^{k-2}-1}-\frac {\alpha^{3}+1}{\alpha^{3}-1}-\frac {\alpha^2+1}{\alpha^2-1}}
{\frac {\alpha^{6}+1}{\alpha^{6}-1} -\frac {\alpha^{3}+1}{\alpha^{3}-1} -\frac {\alpha^{2}+1}{\alpha^{2}-1} -\frac {\alpha+1}{\alpha-1}}
\Big)\\[1em]
=\Big(1,\frac{  (1-\alpha^2)(1+\alpha^3) (1 - \alpha^k) (1 - \alpha^{ k+2})}{(1 - \alpha^5) (1 - \alpha^{
   k - 1}) (1 - \alpha^{k+3})} ,
   \frac{  (1-\alpha+\alpha^2)(1 - \alpha^k) (1 - \alpha^{ k+1})}{(1 + \alpha^2) (1 - \alpha^{
   k - 2}) (1 - \alpha^{k+3})} 
   \Big)
\end{array}
\end{align}
which is exactly the point in \eqref{pkcusp} for the appropriate choice of primitive $n$-th root of $1$.
It then follows from Proposition \ref{set-th} that this is a configuration and that no other cusps map inside the configuration space. (The latter statement can be also seen directly from Proposition \ref{cuspvalues}, we leave the details to the reader).

\smallskip
As a consequence, the realization space $\cR(\cT_n)$ does not contain the special points of \eqref{pkcusp} as isolated points, but rather as part of the smooth image of at open subset of $X_1(n)$. Then the scheme-theoretic equality follows from Proposition \ref{atmost1}.
\end{proof}

\section{Compactifications}\label{Compact}
In this section we discuss two natural compactifications of the configuration space $\cR(\cT_n)$ for $n\geq 10$.

\smallskip
First, observe that if we consider the $\cR(\cT_n)$  as a subset of $M_{n\times 3}(\CC)$, then the closure in the projective space $\PP(M_{n\times 3}(\CC)\oplus \CC)$ is isomorphic to $X_1(n)$. Moreover, we have an embedding given precisely by the weight one forms $s_k(\tau)$. This was initially conjectured by X.R. based on numerical experiments for $10\leq n\leq 42$, and which we can now confirm.

\smallskip
It is also interesting to study the compactification of $\cR(\cT_n)$ in the product of projective planes $(\PP^2)^{n-4}$. Clearly, this means understanding the limits of \eqref{matrixs} as $\tau$ approaches cusps in $X_1(n)$ that are \emph{not} in the $(\ZZ/n\ZZ)^*$ orbit of $\ii\infty$. Such limits exist and are unique by properness of $\PP^2$. If $s_6-s_3-s_2-s_1\neq 0$, then we can get this relatively easily from Proposition \ref{cuspvalues}, but if it is zero, then the calculations are more complicated. In any case, the resulting point configurations will satisfy vanishing of the determinants but not all of the non-vanishings.

\smallskip 
We can state the condition of the cusp $\frac ac$ (with coprime $a$ and $c$) being in the orbit of $\ii\infty$ as $n/\gcd(c,n) = 1$. As such, it is natural to next consider the case $n/\gcd(c,n)=2$, which means that $n=2m$, $c=m\md n$ and $a$ is coprime to $c$.
Then
for even $k$ the value of $s_k$ at this cusp is 
$$
(2\pi\ii) \frac {\zeta^{kd} +1}{2(\zeta^{kd}-1)}
$$
and for odd $k$ the value is $0$. Consequently, $s_6-s_3-s_2-s_1$ takes value
$$
(2\pi\ii) \frac {\zeta^{2d}(1+\zeta^{2d})}{1-\zeta^{6d}}.
$$
For even and odd $k$, the $k$-th row of the matrix \eqref{matrixs} is given by 
$$
\left(
\begin{array}{c}
1,0,
\frac{(1 - \zeta^{6d})(1 - \zeta^{kd})}{ \zeta^{2d} (1 - \zeta^{4d}) (1 - \zeta^{(k-2)d})}
\end{array}
\right)
$$
and
$$
\left(
\begin{array}{c}
1,
\frac{(1 - \zeta^{2d})(1 - \zeta^{6d})  \zeta^{(k-3)d} }{(1 - \zeta^{(k-1)d}) (1 - \zeta^{(k+3)d})},
\frac{(1 - \zeta^{6d}) (1 - \zeta^{(k + 1)d})}{(1 - \zeta^{4d}) (1 - \zeta^{(k + 3)d})}
\end{array}
\right)
$$
respectively. Note that the even $k$ points lie on the line $x_2=0$ and the odd $k$ points lie on the conic 
\begin{align*}
0=(x_3-x_2) x_3 \zeta^{2d} (1+ \zeta^{2d})^2 + x_1^2 (1 +  \zeta^{2d}+  \zeta^{4d})^2 
\\
+ 
 x_1 x_2  \zeta^{2d} (1 +  \zeta^{2d}+  \zeta^{4d}) - x_1 x_3 (1 +  \zeta^{2d})^2 (1 +  \zeta^{2d} +  \zeta^{2d})
 \end{align*}
After a linear change of coordinates, this configuration of points in $\CC\PP^2$ 
can be identified with the B\"or\"oczky example \cite{crowe}, which is important in real arrangements, see \cite{GT}.

\smallskip
We will now consider the next case, that of $n/\gcd(c,n)=3$. In this case we have $n=3m$ and $c$ can be either $m$ or $-m$ modulo $n$. Without loss of generality we may consider the case $c=m$.
Then we have
$$
\lim_{\tau\to\ii\infty}(c\tau+d)^{-1}s_k(\frac{a\tau+b}{c\tau+d})=
\left\{
\begin{array}{ll}
(2\pi\ii)\frac{\zeta^{kd}+1}{2(\zeta^{kd}-1)},&k=0\md 3,\\[.5em]
(2\pi\ii)(-\frac 16),& k=1\md 3,\\[.5em]
(2\pi\ii)(\frac 16),&k=2\md 3.
\end{array}
\right.
$$
Consequently, the $k$-th line of the matrix \eqref{matrixs} is given by 
$$
\left(
\begin{array}{c}
1,
\frac{(1 + \zeta^{3d})  \zeta^{kd} }{ (1 - \zeta^{(k+3)d})},
\frac{(1 + \zeta^{3d})  \zeta^{kd} }{ (1 - \zeta^{(k+3)d})}
\end{array}
\right)
$$
$$
\left(
\begin{array}{c}
1,
\frac{(1 + \zeta^{-3d}) (1 - \zeta^{(k + 2)d})}{ 1 - \zeta^{(k - 1)d}},
1 + \zeta^{3d}
\end{array}
\right)
$$
$$
\left(
\begin{array}{c}
1,
1 + \zeta^{-3d},
\frac{(1 + \zeta^{-3d}) (1 - \zeta^{(k + 1)d})}{ 1 - \zeta^{(k - 2)d}}
\end{array}
\right)
$$
for $k=0,1,2\md 3$ respectively. After a linear change of variables, we can reduce it to
$$
\{(1:0:-\zeta^{3l}),(0:-\zeta^{3l}:1),(-\zeta^{3l}:1:0)\},~0\leq l<m,
$$
which is the Ceva example, (see e.g. \cite{efu}), a prominent example for ball quotient surfaces, see e.g. \cite{Hirzebruch}.

\smallskip
The next case to consider is $n=4m$ and $c=m\md n$. We have
$$
\begin{array}{c}\lim_{\tau\to\ii\infty}(c\tau +d)^{-1}(s_6(\tau)-s_3(\tau)-s_2(\tau)-s_1(\tau))
\\[.5em]=
2\pi\ii\big((\frac 24 -\frac 12) -(\frac 34 -\frac 12)-(\frac 24 -\frac 12)-(\frac 14 -\frac 12)\big) =0
\end{array}
$$
so the previous approach does not work. However, we can use the formulas 
$$a_k = \frac {s_5-2s_3+s_1}{ s_{k+2}-s_k + s_1-s_3 },~b_k= \frac {s_4-2s_3+s_2}{s_{k+1}-s_k + s_2-s_3 }$$
from the proof of Propositions \ref{bkform} and \ref{akform}, at least for some values of $k$. 
The numerator of $a_k$ gives limit
$$
\begin{array}{c}\lim_{\tau\to\ii\infty}(c\tau +d)^{-1}(s_5(\tau)-2s_3(\tau)+s_1(\tau))
\\[.5em]=
2\pi\ii\big((\frac 14 -\frac 12) -2(\frac 34 -\frac 12)+(\frac 14 -\frac 12)\big) =-2\pi\ii.
\end{array}
$$
The numerator of $b_k$ gives the limit $\frac {2\pi\ii}{\zeta^{4d}-1}$, so both are nonzero.
For $k=0\md 4$ and $k=3\md 4$ we get
$(1,a_k,b_k)$ equal $(1,1-\zeta^{-kd},  \frac{1-\zeta^{-kd}}{1-\zeta^{4d}}  )$ and 
$(1,1,\frac{1-\zeta^{(k+1)d}}{1-\zeta^{4d}})$ respectively. For $k=2\md 4$ the denominator of 
$a_k$ equals $\frac {2\pi \ii}{\zeta^{(k+2)d}-1}$, and the denominator of $b_k$ tends to $0$, so the corresponding limit in $\PP^2$ is $(0,0,1)$ for all $k=2\md 4$.  This approach does not work well for the remaining $k=1\md 4$, because in this case both denominators of $a_k$ and $b_k$ go to $0$, so we can only deduce that $p_k$ has zero first coordinate. However, we can still determine these points, since collinearity is preserved in the limit. Indeed, $p_k$ will lie on the line $\overline{p_{n-k}p_0}$. This gives
$$
p_k=(0:1:\frac{1-\zeta^{(1-k)d}}{1 - \zeta^{4d}}).
$$
This is a copy of the Ceva arrangement for $3m$ points and one additional point on the intersection of two lines of the Ceva arrangement, repeated $m$ times.

\smallskip
In the general case, it is relatively easy to see what happens when 
$$
\lim_{\tau\to\ii\infty}(c\tau+d)^{-1}(s_6(\tau)-s_3(\tau)-s_2(\tau)-s_1(\tau))\neq 0.
$$
In this case, the values of $a_k$ and $b_k$ will lie in a finite list, except for when $k\in \{-3,1,2\}\md n/\gcd(c,n)$. Thus we expect it to be again related to the Ceva arrangement. If the above limit is zero, one could use the method of Proposition \ref{cuspvalues} to study the asymptotic behavior of $s_6-s_3-s_2-s_1$ near the cusp, although it does not appear straightforward.

\vspace{3mm}

\noindent \\
Lev Borisov, Rutgers University, 110 Frelinghuysen Rd, Piscataway, NJ 08854, USA

\noindent borisov@math.rutgers.edu

\noindent \\
Xavier Roulleau, Université d'Angers, CNRS, LAREMA, SFR MATHSTIC,
F-49000 Angers, France 

\noindent xavier.roulleau@univ-angers.fr

\end{document}